\documentclass{article}

\usepackage{amsfonts,amssymb,amsmath,amsthm,color,enumerate,graphicx,pictex,stmaryrd,xspace}
\usepackage[mathscr]{eucal}

\newtheorem{example}{Example} 
\newtheorem{question}[example]{Question}
\newtheorem{corollary}[example]{Corollary}
 \newtheorem{lemma}[example]{Lemma}
\newtheorem{theorem}[example]{Theorem}
\newtheorem{proposition}[example]{Proposition}
\newtheorem{conjecture}{Conjecture} 

\newcommand{\Coll}{{\tt Coll}} 
\newcommand{\forces}{\Vdash}

\title{Baire Spaces and Infinite Games} 

\author{Fred Galvin and Marion Scheepers}

\date{\today}

\begin{document}
\maketitle

\begin{abstract}
It is well known that if the nonempty player of the Banach-Mazur game has a winning strategy on a space, then that space is Baire in all powers even in the box product topology. The converse of this implication may also be true: We know of no consistency result to the contrary. In this paper we establish the consistency of the converse relative to the consistency of the existence of a proper class of measurable cardinals.
\footnote{

 $\;$ \lowercase{{\bf {\uppercase{K}}ey words and phrases:} {\uppercase{B}}aire space, infinite game, measurable cardinal}. }
\end{abstract}

A topological space is said to be \emph{Baire} if the intersection of any countable collection of dense open subsets is dense. This property is associated with game theory through the following classical game, called the \emph{Banach-Mazur game}\footnote{For the history of, and a good bibliography for the Banach-Mazur game, see \cite{Telgarsky}.}:

A topological space $X$ is given. Two players, White (or ONE) and Black (or TWO) play the following game: There is an inning per finite ordinal. White starts the game by choosing a nonempty open set $B_0\subseteq X$, and Black responds by choosing a nonempty open set $B_1\subseteq B_0$. In the $n+1^{st}$ inning White chooses a nonempty open set $B_{2n+2}\subseteq B_{2n+1}$, and Black responds with a nonempty open set $B_{2n+3}\subseteq B_{2n+2}$, and so on. In this way the two players produce a \emph{play}
\[
  B_0,\, B_1,\, \cdots,\, B_{2n},\, B_{2n+1},\, \cdots \, (n\in{\omega}).
\]
White wins this play if $\bigcap_{n<{\omega}} B_n = \emptyset$, and Black wins otherwise. The symbol ${\sf BM}(X)$ denotes this game. 
Two of the central results about this game are:
\begin{theorem}[Banach and Mazur, Oxtoby] $X$ is a Baire space if, and only if, White has no winning strategy in ${\sf BM}(X)$. 
\end{theorem}

\begin{theorem}[White \cite{White}, Theorem 3 (4)]\label{boxpowers} If Black has a winning strategy in the Banach-Mazur game on a space, then all powers of that space, considered in the box product topology, are Baire spaces.
\end{theorem}
We conjecture that the converse of Theorem \ref{boxpowers} is true. This conjecture is 
the motivation for the rest of the paper.

In Section 1 of the paper we make the conjecture more precise, and relate it to another conjecture regarding the classical Gale-Stewart game. We formulate a conjecture regarding existence of winning strategies in a variant of the Gale-Stewart game, and show that this second conjecture implies the first one (Theorem \ref{BMGSconnection}). In Section 2 we introduce another class of games, invented by Richard Laver, and show how existence of winning strategies in Laver's game implies the existence of winning strategies in the above variants of the Gale-Stewart game (Theorem \ref{thm:lavertoGS}). In Section 3 we introduce yet another class of games based on the cut-and-choose games introduced by Ulam, and show how existence of winning strategies in these games imply existence of winning strategies in Laver's games (Theorem \ref{thm:ugtolaver}). In Section 4 we study a game on ideals and explain Baumgartner's result that existence of winning strategies in the game on ideals implies the existence of winning strategies in the cut-and-choose games (Theorem \ref{baumgth}). In Section 5 we show how truth of some instances of the original two conjectures can be obtained from a hypothesis about the existence of certain large cardinals. We discuss how the full conjectures 1 and 2 follow from the hypothesis that there is a proper class of large cardinals of precise type. In Section 6  we return to {\sf ZFC} and give a number of examples of the scenario that Conjecture 2 is about. Finally, in Section 7, we state some open problems.

We assume the consistency of {\sf ZFC}, the base theory for this paper. Any other hypotheses or axioms used will be explicitly identified. Also, all products of topological spaces in this paper are assumed to carry the box product topology.

\section{Two conjectures and multiboard games}

A family $\mathcal{B}$ of nonempty open subsets of a topological space $X$ is said to be a $\pi$-\emph{base} if there is for each nonempty open subset $U$ of $X$ an element $V\in\mathcal{B}$ such that $V\subseteq U$. We define the $\pi$-weight, denoted $\pi(X)$, to be the minimal cardinal $\lambda$ such that $X$ has a $\pi$-base of cardinality $\lambda$.

For $\lambda$ and $\kappa$ positive cardinal numbers, we define:
\begin{quote}
 $ {\sf B}(\lambda,\kappa)$: If $X$ is a space with $\pi(X) \le \lambda$ and the power
 $X^{\kappa}$, considered in the box product topology, is Baire, then Black has a winning strategy in the game $\text{BM}(X)$. 
 \end{quote}
Here is a more precise statement of our conjecture: 
\begin{conjecture}\label{galvin1} 
For each $\lambda\ge 1$ there is a $\kappa\ge 1$ such that ${\sf B}(\lambda,\kappa)$ is true.
\end{conjecture}

We are not aware of any (consistent)  examples that disprove Conjecture \ref{galvin1}. 
One of the attractive consequences of Conjecture \ref{galvin1} is:
\begin{proposition}\label{productivity}
Assume Conjecture 1. If $X$ is a Baire space and if $Y$ is a space for which all powers, considered in the box product topology, are Baire spaces, then $X\times Y$ is a Baire space.
\end{proposition}
\begin{proof}
Let $X$ and $Y$ be as given. By Conjecture \ref{galvin1}, TWO has a winning strategy in ${\sf BM}(Y)$. By Theorem 3(9) of \cite{White}, $X\times Y$ is a Baire space.
\end{proof}

We now show that this conjecture follows from another plausible conjecture for which we are also not aware of any (consistent) counterexamples. 

Let a cardinal $\lambda\ge 2$ as well as a set $A\subseteq \,^{\omega}\lambda$ be given. The classical Gale-Stewart game ${\sf GS}_{\lambda}(A)$ between two players, ONE and TWO, is played as follows: The players play an inning per finite ordinal. In inning $n<\omega$ ONE first chooses $\xi_{2n}\in\lambda$, and then TWO responds by choosing $\xi_{2n+1}\in\lambda$. In this way the players construct a play $p = (\xi_n:n<\omega)$. ONE wins this play if $p\in A$, and otherwise, TWO wins the play. 

Let, in addition to the cardinal $\lambda\ge 2$, also a cardinal $\kappa\ge 1$ be given. Then the $\kappa$-board Gale-Stewart game ${\sf GS}_{\lambda}(A,\kappa)$ is defined as follows: Players ONE and TWO play an inning per finite ordinal. In inning $n$ ONE first selects $(f_{2n}(\alpha):\alpha <\kappa) \in \, ^{\kappa}\lambda$, and then TWO responds with $(f_{2n+1}(\alpha):\alpha<\kappa)\in \, ^{\kappa}\lambda$. In this way the players construct a play $(f_n:n<\omega)$, consisting of a sequence of members of $^{\kappa}\lambda$. ONE wins this play if there is an $\alpha<\kappa$ for which $(f_n(\alpha):n<\omega)$ is a member of $A$. Else, TWO wins.

Intuitively, the game $\text{GS}_\lambda(A,\kappa)$ can be interpreted as the two players 
playing the Gale-Stewart game $\text{GS}_\lambda(A)$ on $\kappa$ boards simultaneously. 
The rules are that a player must make a move on each of the boards before the opponent responds by also making a move on each of the boards. ONE wins the multiboard version if there is at least one board on which the play produced on that board is a win for ONE of ${\sf GS}_{\lambda}(A)$. TWO wins the multiboard version if TWO wins each play produced on each of the boards. In this multiboard notation the classical Gale-Stewart game would be denoted ${\sf GS}_{\lambda}(A,1)$. However, for this case we shall continue to write $\text{GS}_{\lambda}(A)$ instead.

Assume that $A\subseteq\,^{\omega}\lambda$ is such that the game ${\sf GS}_{\lambda}(A)$ is undetermined\footnote{It is well known that in {\sf ZFC}  there are such examples.}. Then there are the following two cases for the multiboard version:
\begin{enumerate}
\item[(A)] 
For all $\kappa$, ${\sf GS}_{\lambda}(A,\kappa)$ is undetermined. 
\item[(B)] 
There is a cardinal $\kappa_0> 1$ such that, for all $\kappa$ with $1\le\kappa <\kappa_0$, ${\sf GS}_{\lambda}(A,\kappa)$ is undetermined, but 
${\sf GS}_{\lambda}(A,\kappa_0)$ is determined\footnote{In this case player ONE would have the winning strategy.}. In this case  ${\sf GS}_{\lambda}(A,\kappa)$ is determined for all $\kappa\ge \kappa_0$.
\end{enumerate}

We do not known of a cardinal $\lambda\ge 2$ and a set $A\subseteq\,^{\omega}\lambda$ such that ${\sf GS}_{\lambda}(A,\kappa)$ is undetermined for all $\kappa$. 
We conjecture that case (A) does not occur. More precisely, define:
\begin{quote}
${\sf D}(\lambda,\kappa)$: For each set $A\subseteq \,^{\omega}\lambda$ the game ${\sf GS}_{\lambda}(A,\kappa)$ is determined.
\end{quote}
\begin{conjecture}\label{galvin2} For each cardinal $\lambda\ge 2$ there exists a cardinal $\kappa\ge 1$ such that ${\sf D}(\lambda,\kappa)$ holds.
\end{conjecture}

\begin{theorem}\label{BMGSconnection} Let cardinal numbers $\lambda\ge 2$ and $\kappa\ge 1$ be given. Then ${\sf D}(\lambda,\kappa)$ implies ${\sf B}(\lambda,\kappa)$.
\end{theorem}
\begin{proof}
Let a topological space $X$ of $\pi$-weight $\lambda$ be given, and assume that $\mathcal{B}$ is a $\pi$-base of cardinality $\lambda$ for $X$.
Also assume that in the box product topology $X^{\kappa}$ is a Baire space. We must show that Black has a winning strategy in ${\sf BM}(X)$. 

Identify $\mathcal{B}$ with $\lambda$, and define a subset $A$ of $^{\omega}\mathcal{B}$ to be the set of sequences $(B_n:n<\omega)$ such that:
\begin{enumerate}
\item{$(\forall n)(B_{2n+2}\subseteq B_{2n+1})$;}
\item{$\bigcap_{n<\omega}B_{n} =\emptyset$,  or $(\exists n<\omega)(B_{{2n+1}}\not\subseteq B_{{2n}})$.}
\end{enumerate}
By the hypothesis ${\sf D}(\lambda,\kappa)$, the game ${\sf GS}_{\lambda}(A,\kappa)$ is determined. 

{\flushleft{Claim:}} Player ONE does not have a winning strategy in the game ${\sf GS}_{\lambda}(A,\kappa)$. 

For let $\sigma$ be a strategy for ONE.  Define a corresponding strategy $F$ for White in ${\sf BM}(X^{\kappa})$ as follows:

In inning $0$, $\sigma$ calls on ONE to play an $f_0 = \sigma(\emptyset)$ where $f_0$ is a function $f_0:\kappa\rightarrow\mathcal{B}$. This defines the open subset 
 $B_0 = F\emptyset) := \prod_{\alpha<\kappa} \sigma(\emptyset)(\alpha)$
of $X^{\kappa}$. 

We may assume the players of ${\sf BM}(X^{\kappa})$ play elements of the $\pi$-base obtained by taking products of $\kappa$ elements chosen from the $\pi$-base of $X$.
Now Black responds in ${\sf BM}(X^{\kappa})$ with such an element $B_1\subseteq B_0$. $B_1$ can be described as $B_1:=\prod_{\alpha<\kappa}f_1(\alpha)$, where $f_1\in \,^{\kappa}\mathcal{B}$.

Considering $f_1$ as TWO's response in ${\sf GS}_{\lambda}(A,\kappa)$, apply ONE's strategy there to find $f_2 = \sigma(f_1) \in\,^{\kappa}\mathcal{B}$, and then define White's move in the game ${\sf BM}(X^{\kappa})$ as $B_2 =  F(B_1):=\prod_{\alpha<\kappa}\sigma(f_1)(\alpha)$, and so on, as depicted in the following diagram:

\begin{center}
\begin{figure}[h]
\begin{tabular}{c|l}
\multicolumn{2}{c}{${\sf GS}_{\lambda}(A,\kappa)$}\\ \hline\hline
\multicolumn{2}{c}{}\\
ONE              & TWO                                     \\ \hline
$\sigma(\emptyset)$   &                                       \\
                 & $f_1$ \\
$\sigma(f_1)$   &                                       \\
                 & $f_3 $ \\
                 &                                         \\
$\vdots$         & $\vdots$ \\
\end{tabular}
\hspace{0.1in}
\begin{tabular}{c|l}
\multicolumn{2}{c}{${\sf BM}(X^{\kappa})$}                        \\ \hline\hline
\multicolumn{2}{c}{}\\ 
White   & Black                                                 \\ \hline
$F(\emptyset) = \prod_{\alpha<\kappa}\sigma(\emptyset)(\alpha)$ & \\
                                                                                                        & $B_1 = \prod_{\alpha<\kappa}f_1(\alpha)$                                      \\
$F(B_1) = \prod_{\alpha<\kappa}\sigma(f_1)(\alpha)$      &                                                     \\
      & $B_3=\prod_{\alpha<\kappa}f_3(\alpha)$                                    \\
      &                                                     \\
$\vdots$         & $\vdots$ \\
\end{tabular}
\caption{Playing a Banach-Mazur game using a multiboard Gale-Stewart game}\label{fig: BMGS}
\end{figure}
\end{center}

Since $X^{\kappa}$ is a Baire space, White has no winning strategy in ${\sf BM}(X^{\kappa})$. But then consider an $F$-play of ${\sf BM}(X^{\kappa})$ lost by White:
\[
  F(X^{\kappa}),\, B_1,\, F(B_1),\, B_3,\, \cdots, F(B_1,\cdots,B_{2n-1}),\, B_{2n+1},\, \cdots
\]
We have for each $n$ that
\begin{itemize}
\item{$B_{2n+1} = \prod_{\alpha<\kappa}f_{2n+1}(\alpha)$ where $f_{2n+1}$ is a function from $\kappa$ to $\mathcal{B}$, and }
\item{$F(B_1,\cdots,,\,B_{2n-1}) = \prod_{\alpha<\kappa}\sigma(f_1,\cdots,\,f_{2n-1})(\alpha)$.}
\end{itemize}
Since Black wins this play of ${\sf BM}(X^{\kappa})$ it follows that $\bigcap_{n<\omega}B_n \neq \emptyset$. This in turn implies that for each $\alpha<\kappa$ we have that $\bigcap_{n<\omega}f_{2n+1}(\alpha) \neq \emptyset$. But then the $\sigma$-play
\[
 \sigma(\emptyset),\, f_1,\, \sigma(f_1),\, f_3,\, \sigma(f_1,f_3),\, \cdots
\]
of ${\sf GS}_{\lambda}(A,\kappa)$ is lost by ONE, completing the proof of the claim.

Thus, ONE has no winning strategy in the game ${\sf GS}_{\lambda}(A,\kappa)$. By ${\sf D}(\lambda,\kappa)$ TWO has a winning strategy in 
${\sf GS}_{\lambda}(A,\kappa)$, and thus also has a winning strategy, say $\tau$, in ${\sf GS}_{\lambda}(A)$. Now use the winning strategy $\tau$ of TWO to define a winning strategy $F$ for Black in the Banach-Mazur game on $X$:

When White plays  $B_{0}$ in ${\sf BM}(X)$, Black responds with $F(B_{0}):= \tau(B_{0})$. Upon White's next move $B_{2}$ Black responds with $F(B_{0},\, B_{2}):= \tau(B_{0},\, B_{2})$, and so on. As $\tau$ is a winning strategy for TWO in ${\sf GS}_{\lambda}(A)$ it follows that $\bigcap_{n<\omega}B_{n}$ is nonempty, and thus $F$ is a winning strategy for Black in ${\sf BM}(X)$.
\end{proof}

\section{Laver games}

We now introduce a class of games invented by Richard Laver. 
For cardinals $\lambda \ge 2$ and $\kappa \ge 1$, the Laver game ${\sf LG}(\lambda, \kappa)$ is as follows: Player ONE chooses $f_0: \kappa \rightarrow \lambda$, then player TWO chooses $f_1: \kappa \rightarrow \lambda$, then player ONE chooses $f_2: \kappa \rightarrow \lambda$, and so on. 
ONE wins the play $(f_0,\, f_1,\, f_2,\, \cdots)$ of ${\sf LG}(\lambda,\kappa)$ if for each $S\subseteq\,^{\omega}\lambda$ such that TWO does not have a winning strategy in ${\sf GS}_{\lambda}(S)$, 
there exists $\xi \in \kappa$ such that $(f_0(\xi), f_1(\xi), f_2(\xi), . . .)\in S$.
Otherwise, TWO wins.

\begin{lemma}\label{lem:LGequivalents}
Let cardinals $\lambda\ge 2$ and $\kappa\ge 1$, and a sequence $(f_n:n<\omega)$ of elements of $^{\kappa}\lambda$ be given. Define $Y = \{(f_n(\xi):n<\omega): \xi\in\kappa\}$. The following are equivalent:
\begin{enumerate}
\item[(a)] For each $S\subseteq\,^{\omega}\lambda$ such that TWO does not have a winning strategy in ${\sf GS}_{\lambda}(S)$, 
there exists $\xi \in \kappa$ such that $(f_0(\xi), f_1(\xi), f_2(\xi), . . .)\in S$;
\item[(b)] TWO has a winning strategy in 
${\sf GS}_{\lambda}(^{\omega}\lambda \setminus Y)$; 
\item[(c)] There is strategy $\sigma$ for player TWO in the game structure (meaning that the winning condition is unspecified) ${\sf GS}_{\lambda}$ such that every possible $\sigma$-play of ${\sf GS}_{\lambda}$ is in $Y$.
\end{enumerate}
\end{lemma}
\begin{proof}
$(a)\Rightarrow(b):$ Assume (a). Put $S = \,^{\omega}\lambda\setminus Y$. Suppose that contrary to (b), TWO does not have a winning strategy in ${\sf GS}_{\lambda}(S)$.  
By (a) it follows that $Y\cap S$ is nonempty, a contradiction.\\
$(b)\Rightarrow(c):$ Assume (b) and let $\sigma$ be a winning strategy for TWO in 
${\sf GS}_{\lambda}(\,^{\omega}\lambda\setminus Y)$.
Then $\sigma$ is a strategy as in (c) for TWO.\\
$(c)\Rightarrow(a):$ Let $\sigma$ be a strategy for TWO in the game structure ${\sf GS}_{\lambda}$ as in (c). Consider any set $S\subseteq \,^{\omega}\lambda$ for which TWO does not have a winning strategy in ${\sf GS}_{\lambda}(S)$. 
Now let TWO use the strategy $\sigma$ to play this game. Since $\sigma$ is not a winning strategy for TWO, there is a play against $\sigma$ that is won by ONE, say $(\nu_0,\, \nu_1,\, \cdots,\,\nu_{2n},\, \nu_{2n+1},\cdots)$. Since ONE wins this play, $(\nu_0,\, \nu_1,\, \cdots,\,\nu_{2n},\, \nu_{2n+1},\cdots) \in S$. Since this play is a $\sigma$-play, it is according to (c) a member of $Y$, and thus of the form $(f_n(\xi):n<\omega)$ for some $\xi\in\kappa$.
\end{proof}

Let ${\sf L}(\lambda,\kappa)$ denote the statement: ONE has a winning strategy in ${\sf LG}(\lambda,\kappa)$.

\begin{theorem}\label{thm:lavertoGS} For cardinal numbers $\lambda$ and $\kappa$, 
${\sf L}(\lambda,\kappa)$ implies ${\sf D}(\lambda,\kappa)$. 
\end{theorem}
\begin{proof}
Let $\sigma$ be a winning strategy for ONE in 
${\sf LG}(\lambda,\kappa)$. Let $S\subseteq\,^{\omega}\lambda$ be given. We must show that ${\sf GS}_{\lambda}(S,\kappa)$ is determined. If TWO has a winning strategy in ${\sf GS}_{\lambda}(S)$ 
there is nothing to prove. Thus, assume that TWO does not have a winning strategy in ${\sf GS}_{\lambda}(S)$.

Now ONE of ${\sf GS}_{\lambda}(S,\kappa)$ uses the strategy $\sigma$ of ONE of ${\sf LG}(\lambda,\kappa)$ to play ${\sf GS}_{\lambda}(S,\kappa)$. Let $(f_0,f_1,\cdots,f_{2n},f_{2n+1},\cdots)$ be a $\sigma$-play of ${\sf GS}_{\lambda}(S,\kappa)$. Since this play is a winning play for ONE of  
${\sf LG}(\lambda,\kappa)$, and since TWO does not have a winning strategy in $\textsf{GS}_{\lambda}(S)$,  
there exists $\xi\in\kappa$ for which $(f_n(\xi):n<\omega)$ is in $S$. But then this $\sigma$-play is also a winning play for ONE in the game ${\sf GS}_{\lambda}(S,\kappa)$.
\end{proof}

We leave the proof of the following partial converse to Theorem \ref{thm:lavertoGS} to the reader:

\begin{theorem}\label{thm:GStolaver} If ${\sf D}(\lambda, \kappa)$, then ${\sf L}(\lambda, 2^{\lambda^{\aleph_0}} \cdot \kappa)$.
\end{theorem}

Thus, when $\kappa\ge 2^{\lambda^{\aleph_0}}$, the games ${\sf D}(\lambda,\kappa)$ and ${\sf L}(\lambda,\kappa)$ are equivalent. This suggests that ${\sf L}(\lambda, \kappa)$ is not much (if at all) stronger than ${\sf D}(\lambda, \kappa)$. 

\begin{theorem}\label{thm:TWOlaver}  TWO has a winning strategy in ${\sf LG}(2,2^{\aleph_0})$.
\end{theorem}

\begin{proof} Let the set of ``boards" be some totally imperfect subset $X$ of the Cantor space $\ ^\omega\{0,1\}$ with $|X|=2^{\aleph_0}$; thus, at move $n$, ONE chooses a function $f_n:X\to\{0,1\}$, and then TWO chooses a function $g_n:X\to\{0,1\}$. TWO's winning strategy $\pi$ is very simple: at move $n$ he chooses the function $g_n(x)=x(n)$. 
Assume towards a contradiction that ONE wins some $\pi$-play $(f_0,g_0,f_1,g_1,\dots)$ of ${\sf LG}(2,2^{\aleph_0})$, and define
$$Y=\{(f_0(x),g_0(x),f_1(x),g_1(x),\dots):x\in X\}.$$
Then, by the implication (a) $\Rightarrow$ (c) of Lemma 5, there is a strategy $\sigma$ for TWO in ${\sf GS}_2$ such that every $\sigma$-play of 
${\sf GS}_2$ is in $Y$.

For each $t\in\ ^\omega\{0,1\}$ let $y_t\in Y$ be the $\sigma$-play of ${\sf GS}_2$ in which $t_n$ is the sequence of ONE's moves, and let $x_t\in X$ be such that
\begin{center}
\begin{tabular}{lcl}
$y_t$ & = & $(f_0(x_t),g_0(x_t),f_1(x_t),g_1(x_t),\dots)$ \\
          & = & $(t(0),x_t(0),t(1),x_t(1),\dots)$ \\
          & = & $(t(0),\sigma(t(0)),t(1),\sigma(t(0),t(1)),\dots).$\\
 \end{tabular}
\end{center}
Then the function $t\mapsto x_t$ is a one-to-one continuous mapping of $\ ^\omega\{0,1\}$ into $X$; but this is impossible because $X$ is totally imperfect.
\end{proof}

\section{Cut-and-choose games and the conjectures}

Again, let $\lambda\ge 2$ and $\kappa\ge 1$ be cardinals. Consider the following
cut-and-choose game ${\sf UG}(\lambda,\kappa)$  
between players White and Black\footnote{The game ${\sf UG}(2,\kappa)$ was invented by Ulam: See pp. 346-347 of \cite{ulam}.}:

First, White chooses a partition of $\kappa$ into $\lambda$ pieces, and
then Black chooses one of these, $S_0$. Then Black chooses a partition of
$S_0$ into $\lambda$ pieces, and White chooses one of these, $S_1$. Then White
chooses a partition of $S_1$ into $\lambda$ pieces, and Black chooses one
of these, say $S_2$, and so on for $\omega$ innings. White wins if $\bigcap_{n<\omega}S_n\neq \emptyset$. Else, Black wins. 

Alternately this game can be described as follows: White chooses a function $f_0 \in\,^{\kappa}\lambda$. Then Black chooses a $\xi_0\in\lambda$. Then Black chooses a $f_1\in\, ^{\kappa}\lambda$, and then White chooses a $\xi_1\in\lambda$. Then White chooses an $f_2 \in\,^{\kappa}\lambda$. Then Black chooses a $\xi_3\in\lambda$, and an $f_3\in\,^{\kappa}\lambda$, and so on. A play
\[
  f_0,\, (\xi_0,f_1),\,  (\xi_1,f_2),\, (\xi_2, f_3),\, \cdots
\]
is won by White if there is an $\alpha<\kappa$ such that for each $n$, $f_n(\alpha) = \xi_n$.
We define:
\begin{quote}
${\sf U}(\lambda,\kappa)$: White has a winning strategy in ${\sf UG}(\lambda,\kappa)$.
\end{quote}
The following observation will be useful later:
\begin{lemma}\label{monotone} If $\lambda, \, \lambda^{\prime},\, \kappa$ and $\kappa^{\prime}$ are cardinal numbers such that  $1\le \lambda^{\prime}\le \lambda$ and $\kappa^{\prime}\ge \kappa\ge 1$, then ${\sf U}(\lambda,\kappa)$ implies ${\sf U}(\lambda^{\prime},\kappa^{\prime})$.
\end{lemma}

Here is how this cut-and-choose game is related to the Laver game:
\begin{theorem}\label{thm:ugtolaver}
For all cardinals $\lambda\ge 2$ and $\kappa\ge 1$, ${\sf U}(\lambda, \kappa)$ implies ${\sf L}(\lambda, \kappa)$.
\end{theorem}
\begin{proof}
Let $\lambda\ge 2$ and $\kappa\ge 1$ be given cardinal numbers. Assume that ${\sf U}(\lambda,\kappa)$ holds. We must show that ${\sf L}(\lambda,\kappa)$ holds.

Recall that a play of ${\sf LG}(\lambda,\kappa)$ is a sequence
$(f_0,f_1,f_2,\dots)$ in $^\kappa\lambda$ where $f_0,f_2,f_4,\dots$ are
chosen by ONE and $f_1,f_3,f_5,\dots$ by TWO. By Lemma \ref{lem:LGequivalents}, ONE wins the play
$(f_0,f_1,f_2,\dots)$ just in case there is a strategy $\sigma$ for TWO in
$\text{GS}_\lambda$ such that every $\sigma$-play of ${\sf GS}_{\lambda}$
has
the form $(f_0(\xi),f_1(\xi),f_2(\xi),\dots)$ for some $\xi\in\kappa$.

Let $\varphi$ be a winning strategy for White 
in ${\sf UG}(\lambda,\kappa)$, for which we will use the following notation. On the first move, White plays $f_0^\varphi$. For $n> 0$, if Black's first $n$ moves are
$(\alpha_0,f_1),\dots,(\alpha_{2n-2},f_{2n-1})$, White's next move is
\[
  (\alpha_{2n-1}^\varphi((\alpha_0,f_1),\dots,(\alpha_{2n-2},f_{2n-1})),
f_{2n}^\varphi((\alpha_0,f_1),\dots,(\alpha_{2n-2},f_{2n-1}))).
\]

Now we define a strategy $\Phi$ for ONE in ${\sf LG}(\lambda,\kappa)$. It
will be convenient to describe it in the form
$f_{2n}=\Phi(f_0,f_1,f_2,\dots,f_{2n-1})$, i.e., ONE's next move is given as
a function of the total history, including the moves of both players.
Namely, we define $\Phi(\emptyset)=f_0^\varphi$, and for $n> 0,\, \xi\in\kappa$,
we define
\[
  \Phi(f_0,f_1,\dots,f_{2n-1})(\xi)=f_{2n}^\varphi((f_0(\xi),f_1),(f_2(\xi),f_3),\dots,(f_{2n-2}(\xi),f_{2n-1}))(\xi).
\]

Suppose $(f_0,f_1,f_2,\dots)$ is a $\Phi$-play of
${\sf LG}(\lambda,\kappa)$; we have to show that this play is won by ONE.
First, we define a strategy $\sigma$ for TWO in ${\sf GS}_\lambda$, by
setting
\[
   \sigma(\alpha_0,\alpha_2,\dots,\alpha_{2n})=\alpha_{2n+1}^\varphi((\alpha_0,f_1),(\alpha_2,f_3),\dots,(\alpha_{2n},f_{2n+1})).
\]
Now let $(\alpha_0,\alpha_1,\alpha_2,\dots)$ be any $\sigma$-play of
$\text{GS}_\lambda$; we have to find $\xi\in\kappa$ such that
$f_n(\xi)=\alpha_n$ for all $n$.

For $n<\omega$ let $\hat
f_{2n}=f_{2n}^\varphi((\alpha_0,f_1),(\alpha_2,f_3),\dots,(\alpha_{2n-2},f_{2n-1}))$.
Then 
\[  (\hat f_0,(\alpha_0,f_1),(\alpha_1,\hat f_2),(\alpha_2,f_3),(\alpha_3,\hat f_4),\dots)
\] 
is a $\varphi$-play of
${\sf UG}(\lambda,\kappa)$. Since $\varphi$ is a winning strategy for White in ${\sf UG}(\lambda,\kappa)$, there 
exists $\xi\in\kappa$ such that $\hat f_{2n}(\xi)=\alpha_{2n}$ and $f_{2n+1}(\xi)=\alpha_{2n+1}$ for all $n< \omega$.

Finally, we prove by induction on $n$ that $f_{2n}(\xi)=\alpha_{2n}$ for all $n<\omega$. This is clear for $n=0$, since $f_0=f_0^\varphi=\hat f_0$. Now suppose that $n>0$, and $f_{2k}(\xi)=\alpha_{2k}$ for all $k< n$; then
\[
  \begin{tabular}{lcl}
          $f_{2n}(\xi)$ & = & $\Phi(f_0,f_1,f_2,\dots,f_{2n-1})(\xi)$ \\
                               & = & $f_{2n}^\varphi((f_0(\xi),f_1),(f_2(\xi),f_3),\dots,(f_{2n-2}(\xi),f_{2n-1}))(\xi)$ \\
                               & = & $f_{2n}^\varphi((\alpha_0,f_1),(\alpha_1,f_3),\dots,(\alpha_{2n-2},f_{2n-1}))(\xi)$\\
                               & = & $\hat f_{2n}(\xi)=\alpha_{2n}$.
  \end{tabular}
\]
\end{proof}
\begin{theorem}\footnote{Theorem \ref{thm:continuum} was independently discovered by S. Hechler.}\label{thm:continuum} 
If White has a winning strategy in ${\sf UG}(2,\kappa)$, then $\kappa>2^{\aleph_0}$.
\end{theorem}
\begin{proof} By Theorem \ref{thm:ugtolaver}, if White has a winning strategy in $\textsf{UG}(2,\kappa)$, then ONE has a winning strategy in $\textsf{LG}(2,\kappa)$. Then Theorem \ref{thm:TWOlaver} implies that $\kappa>2^{\aleph_0}$.
\end{proof}

\section{$\lambda=\aleph_0$ and Baumgartner's condition}

For cardinals $\lambda\ge2$ and $\kappa \ge 1$ we now have the following implications:
\[
  {\sf U}(\lambda,\kappa) \stackrel{\mbox{\tiny Th \ref{thm:ugtolaver}}}{\Longrightarrow} {\sf L}(\lambda,\kappa)\stackrel{\mbox{\tiny Th \ref{thm:lavertoGS}}}{\Longrightarrow}  {\sf D}(\lambda,\kappa)\stackrel{\mbox{\tiny Th \ref{BMGSconnection}}}{\Longrightarrow} {\sf B}(\lambda,\kappa)
\]
This sequence of implications gives us deeper insight into Conjecture \ref{galvin1} and Conjecture \ref{galvin2}. Before considering the general cases of these conjectures in the next section, we 
first explore the special case when $\lambda=\aleph_0$. The results in this section explore the constraints on $\kappa$ under which White has a winning strategy in the game ${\sf UG}(\aleph_0,\kappa)$. 

\begin{theorem}\label{thm:omega2equivalence} 
Let $\kappa$ be an infinite cardinal number. Then White has a winning strategy in ${\sf UG}(\aleph_0,\kappa)$ if, and only if, White has a winning strategy in ${\sf UG}(2,\kappa)$.
\end{theorem}
\begin{proof} We show that if White has a winning strategy in ${\sf UG}(2,\kappa)$, then White has a winning strategy in ${\sf UG}(\aleph_0,\kappa)$.

Let $\sigma$ be a winning strategy for White in ${\sf UG}(2, \kappa)$. If $p$ is a partial
$\sigma$-play of ${\sf UG}(2, \kappa)$, let $\mathcal{S}(p)$ be the collection of all sets $S \subseteq \kappa$ such that some subset of $S$ is chosen in some $\sigma$-play $p^{\prime}$ which is a continuation of $p$.

{\flushleft{\underline{Claim 1.}}} Any set $S$ in $\mathcal{S}(p)$ can be partitioned into $\aleph_0$ pairwise disjoint sets, each of which belongs to $\mathcal{S}(p)$.

{\flushleft Proof of Claim 1.} It is enough to show that each member of $\mathcal{S}(p)$ contains two disjoint members of $\mathcal{S}(p)$. Suppose $S\in\mathcal S(p)$, and let the $\sigma$-play $p'$ be a continuation of $p$ in which White chooses a set $S_i\subseteq S$. After choosing the set $S_i$ White cuts it into two pieces, both of which are of course members of $\mathcal S(p)$.

{\flushleft{\underline{Claim 2.}}} If $\bigcup_{n < \omega}S_n = S \in \mathcal{S}(p)$, then $S_n\in\mathcal{S}(p)$ for
some $n$.

{\flushleft Proof of Claim 2} Assume the contrary, $S_n \not \in \mathcal{S}(p)$ for all $n$. Starting from $p$, since $S$ is in 
$\mathcal{S}(p)$, Black has a sequence of moves which forces White (using $\sigma$) to choose 
a subset of $S$. Next, by presenting White with the choice of a set contained in $S_0$ 
and a set disjoint from $S_0$, Black forces White to choose a set disjoint from $S_0$. 
Continuing in this way, he forces White to choose a set disjoint from $S_n$ for every $n$. 
Thus Black wins a $\sigma$-play of ${\sf UG}(2, \kappa)$, contradicting the assumption that 
$\sigma$ is a winning strategy for White.\\

To win ${\sf UG}(\aleph_0, \kappa)$, White starts by cutting $\kappa$  
into $\aleph_0$ pieces, each of which belongs to $\mathcal{S}(\emptyset)$. Black chooses one of them, call it $S_0$, and White chooses a partial $\sigma$-play $p_0$ of ${\sf UG}(2, \kappa)$ which ends up at a subset of $S_0$. Next, Black cuts $S_0$ into $\aleph_0$ pieces. Since $S_0$ is in $\mathcal{S}(p_0)$, White can choose a piece $S_1$ which belongs to $\mathcal{S}(p_0)$. White extends $p_0$ to a partial 
$\sigma$-play $p_1$ which ends up at a subset of $S_1$, cuts it into $\aleph_0$ pieces each belonging to $\mathcal{S}(p_1)$, and so on. Since White wins the imagined $\sigma$-play of ${\sf UG}(2, \kappa)$, White also wins the play of ${\sf UG}(\aleph_0, \kappa)$.
\end{proof}

Recall that a cardinal $\kappa$ is \emph{measurable} if it is uncountable and there is a $\kappa$-complete nonprincipal ultrafilter $\mathcal{U}\subseteq\mathcal{P}(\kappa)$. See Jech \cite{jech} for more information on these notions. 
C. Gray and R.M. Solovay proved (unpublished) that if it is consistent that there is an infinite cardinal number $\kappa$ such that White has a winning strategy in the game ${\sf UG}(2,\kappa)$, then it is consistent that there is a measurable cardinal. M. Magidor proved that if it is consistent that there is a measurable cardinal, then it is consistent that there is an infinite cardinal $\kappa$ such that White has a winning strategy in the game ${\sf UG}(2,\kappa)$. 

M. Magidor (unpublished) also showed that if it is consistent that there is a measurable cardinal $\kappa$, then it is consistent that $\kappa$ is measurable and White has a winning strategy in ${\sf UG}(2,\kappa)$. Then R.M. Solovay (unpublished) proved that if it is consistent that there is a measurable cardinal $\kappa$, then it is consistent\footnote{The model is of the form ${\sf L}\lbrack \mathcal{U}\rbrack$ where $\mathcal{U}$ witnesses the measurability of $\kappa$.} that $\kappa$ is measurable, but White does not have a winning strategy in the game ${\sf UG}(2,\kappa)$.  Then Laver proved that if it is consistent that there is a measurable cardinal then it is consistent that there is a successor $\kappa$ of an infinite regular cardinal such that White has a winning strategy in the game ${\sf UG}(2,\kappa)$. Laver's argument went unpublished also, but subsequently appeared in \cite{GJM}. We now give an exposition of ideas of Baumgartner and of Laver that culminates in the consistency, under appropriate large cardinal hypotheses, of Conjectures \ref{galvin1} and \ref{galvin2}.

Let $S$ be a set. A nonempty family $J\subseteq{\mathcal P}(S)$ is said to be an \emph{ideal} on $S$ if: $A\cup B$ is an element of $J$ whenever $A$ and $B$ are elements of $J$; If $B\in J$ and $A\subseteq B$ then $A\in J$. $J$ is said to be a \emph{proper} ideal if $J\neq {\mathcal P}(S)$. $J$ is said to be a \emph{free} ideal if $\bigcup J = S$.  For an ideal $J$ on a set $S$, the symbol $J^+$ denotes the set $\mathcal{P}(S)\setminus J$. The ideal $J\subseteq{\mathcal P}(S)$ is said to be \emph{atomless} if for each $X\in J^+$ there is a partition $X = A\cup B$ such that $A$ and $B$ are disjoint, and both are elements of $J^+$.

For $J$ a proper ideal on set $S$ the game ${\sf G}(J)$ of length $\omega$ is played as follows: In the first inning White chooses a set $W_0\subseteq S$ with $W_0\in J^+$, and Black responds with $B_0\subseteq W_0$ and $B_0\in J^+$. In the $n+1^{st}$ inning White chooses $W_{n}\subseteq B_{n-1}$ with $W_{n}\in J^+$, and Black responds with $B_{n}\subseteq W_{n}$ and $B_{n}\in J^+$. A play is won by Black if $\cap_{n<\omega}B_n\neq\emptyset$; else, White wins. This game was studied in \cite{GJM}.
 Banach considered the special case of $J=\{X\subseteq S\colon\vert X\vert<\vert S\vert\}$. Banach's game is recorded in Problem 67 of the Scottish Book \cite{scottishbook}. Schreier \cite{Schreier} showed that White has a winning strategy in Banach's game.

\begin{theorem}[Baumgartner]\label{baumgth} If there is an atomless proper ideal $J$ on $\kappa$ such that Black has a winning strategy in ${\sf G}(J)$, then White has a winning strategy in ${\sf UG}(\aleph_0,\kappa)$.
\end{theorem}
\begin{proof}
Assume that $\kappa$ is an infinite cardinal number and that $J\subset{\mathcal P}(\kappa)$ is an atomless proper ideal on $\kappa$ for which Black has a winning strategy in the game ${\sf G}(J)$. We must show that then White has a winning strategy in the game ${\sf UG}(\aleph_0,\kappa)$. By Theorem \ref{thm:omega2equivalence} it is sufficient to show that White has a winning strategy in ${\sf UG}(2,\kappa)$.

Let $F$ be a winning strategy for Black in the game ${\sf G}(J)$. Define a strategy for White in the game ${\sf UG}(2,\kappa)$ as follows: 

Since $J$ is atomless, White's first move is $(C_0, D_0)$ where both $C_0$ and $D_0$ are in $J^+$, are disjoint, and have union $\kappa$. 

Suppose that Black selects one of these, $S_0$, and partitions it into two disjoint sets $(C_1, D_1)$. To decide White's response to Black's move, do the following: Since $S_0$ is in $J^+$ select $S_1\in\{C_1,\, D_1\}$ in $J^+$ and assign $B_0=S_1$ to White of the game ${\sf G}(J)$. 
Compute the response $B_1:=F(B_0)$ of Black of the game ${\sf G}(J)$, using Black's winning strategy $F$. 

Since $B_1$ is in $J^+$ and $B_1 \subseteq B_0$ and $J$ is atomless, choose a partition $(C_2, D_2)$ of $S_1$ with both $C_2\cap B_1$ and $D_2\cap B_1$ in $J^+$.

Suppose that Black chooses $S_2\in \{C_2, D_2\}$, and partitions it into two disjoint sets $(C_3,D_3)$.
To decide White's response to Black's move in the game ${\sf UG}(2,\kappa)$, do the following: Since $S_2\cap B_1$ is in $J^+$, select $S_3 \in \{C_3,\, D_3\}$ with $S_3\cap B_1 \in J^+$. Assign $B_2 = S_3\cap B_1$ to White of the game ${\sf G}(J)$. 
Compute the response $B_3:=F(B_0,B_2)$ of Black in the game ${\sf G}(J)$. 

As $J$ is an atomless ideal and $B_3$ is in $J^+$ White of the game ${\sf UG}(2,\kappa)$ chooses a partition $(C_4,D_4)$ of $S_3$ such that both $C_4\cap B_3$ and $D_4\cap B_3$ are in $J^+$, and so on.

\begin{center}
\begin{figure}[h]
\begin{tabular}{c|l}
\multicolumn{2}{c}{${\sf UG}(2,\kappa)$}\\ \hline\hline
\multicolumn{2}{c}{}\\
White                              & Black                                     \\ \hline
$\kappa$; $(C_0,D_0)$ &                                        \\
                                      & $S_0$; $(C_1,D_1)$      \\
\hspace{0.3in} $S_1$; $(C_2,D_2)$                 &                                        \\
                                      & $S_2$; $(C_3,D_3)$      \\
\hspace{0.3in}$S_3$; $(C_4,D_4)$   &                                         \\
                                      &  $S_4$; $(C_5,D_5)$                                         \\
\hspace{0.3in}$S_5$; $(C_6,D_6)$   &                                         \\
                                      &                                       \\
%

$\vdots$                         & $\vdots$ \\
\end{tabular}
\hspace{0.1in}
\begin{tabular}{c|l}
\multicolumn{2}{c}{${\sf G}(J)$}                        \\ \hline\hline
\multicolumn{2}{c}{}\\ 
White                             & Black               \\ \hline
                                     &                                       \\
                                     &                                                     \\
$B_0 = S_1$ & \\ 
                                     & $B_1:=F(B_0)$ \\
 $B_2 = S_3\cap B_1$ & \\ 
                                    & $B_3:=F(B_0,B_2)$                                    \\
$B_4 = S_5\cap B_3$ & \\ 
                                    & $B_5:=F(B_0,B_2,B_4)$                                    \\
$\vdots$                      & $\vdots$ \\
\end{tabular}
\caption{Playing ${\sf UG}(2,\kappa)$ using ${\sf G}(J)$.}\label{fig: Baumgartner}
\end{figure}
\end{center}

To see that this strategy for White in the game ${\sf UG}(2,\kappa)$ is a winning strategy, consider a play according to the strategy:
\[
 (C_0,D_0),\, (S_0;\, (C_1,\, D_1)),\, (S_1;\,(C_2,\, D_2)),\, (S_2;\, (C_3,\, D_3)),\,(S_3,\, (C_4,D_4),\, \cdots  
\]
Corresponding to this play there is an $F$-play $B_0,B_1,B_2,B_3,\dots$ of ${\sf G}(J)$ where $B_{2n}\subseteq S_{2n+1}$ for all $n$. Since $F$ is a winning strategy for Black in ${\sf G}(J)$ we have $\bigcap_{n<\omega}S_n\supseteq\bigcap_{n<\omega}B_n\ne\emptyset$, whence the strategy for White in ${\sf UG}(2,\kappa)$ is a winning strategy.
\end{proof}

We shall say that an 
ideal $J\subseteq\mathcal{P}(\kappa)$ satisfies \emph{Baumgartner's condition} if it is an atomless proper ideal for which Black has a winning strategy in the game ${\sf G}(J)$.
A free ideal $J$ on a set $S$ is said to be \emph{precipitous} if White does not have a winning strategy in the game ${\sf G}(J)$\footnote{The notion of a precipitous ideal has another definition that was shown in Theorem 2 of \cite{GJM} to be equivalent to White not having a winning strategy in the game {\sf G}(J).}. Thus, ideals satisfying Baumgartner's condition are precipitous ideals.

\begin{theorem}\label{thm:ConBaumgCondition} If it is consistent that there is an atomless ideal $J$ on a cardinal $\kappa$ such that Black has a winning strategy in the game ${\sf G}(J)$, then it is consistent that there is a measurable cardinal. 
\end{theorem}
\begin{proof}
If $\kappa$ is a cardinal and $J$ is an atomless ideal on $\kappa$ such that Black has a winning strategy in ${\sf G}(J)$, then $J\subseteq\mathcal{P}(\kappa)$ is a precipitous ideal. Thus, by Lemmas 2.1 and 2.2 of \cite{JMP} , there is an inner model in which there is a measurable cardinal. 
\end{proof}

We now describe a class of ideals that satisfy Baumgartner's condition.   When $J\subseteq \mathcal{P}(\kappa)$ is a free ideal, a family $\mathcal{F}\subseteq J^+$ is said to be \emph{dense} if there is for each element $X$ of $J^+$ an element $Y$ of $\mathcal{F}$ such that $Y\subseteq X$. 

\begin{theorem}\label{thm:densefamilies} Let $J\subseteq \mathcal{P}(\kappa)$ be a proper free ideal. If there is a dense family $\mathcal{F}\subseteq J^+$ with the property that $\bigcap_{n<\omega}X_n\neq \emptyset$ for any sequence $(X_n:n<\omega)$ of elements of $\mathcal{F}$ such that for each $n$, $X_{n+1}\subseteq X_n$, then Black has a winning strategy in the game ${\sf G}(J)$.
\end{theorem} 
\begin{proof}
Black, in fact, has a very simple winning strategy, namely when White plays a set $X\in J^+$, then Black responds with $F(X)\in\mathcal{F}$ such that $F(X)\subseteq X$.
\end{proof}

Laver proved a result which implies
\begin{theorem}[Laver]\label{thm:Laverdenseideal}
If it is consistent that there is a measurable cardinal $\kappa$ then it is consistent that {\sf CH} holds and there is an atomless ideal $J$ on $\aleph_2$ and a dense family $\mathcal{F}\subseteq J^+$ such that for each sequence $(X_n:n<\omega)$ of sets in $J^+$ with $X_n\supseteq X_{n+1}$ for all $n$, the set $\bigcap_{n<\omega}X_n$ is in $J^+$.
\end{theorem}
We shall explain a proof of Laver's more general version of this result later in the paper.

As a result of the facts given in this section of the paper we can now conclude:
\begin{theorem}\label{thm:aleph0conj}
If it is consistent that there is a measurable cardinal, then it is consistent that $2^{\aleph_0} = \aleph_1$ and ${\sf U}(\aleph_0,\aleph_2)$ holds.
\end{theorem}
Thus, the consistency of the existence of a measurable cardinal implies the consistency of the each of the special cases ${\sf B}(\aleph_0,(2^{\aleph_0})^+)$ and ${\sf D}(\aleph_0,(2^{\aleph_0})^+)$ of Conjectures \ref{galvin1} and \ref{galvin2}, respectively. 

\section{The general case of the conjectures}

Let $\kappa$ and $\lambda$ be infinite cardinals. Define:
 \begin{quote} $\textsf{I}(\lambda,\kappa)$: There is a free proper ideal ${J}$ on $\kappa$ such that
   \begin{enumerate}
   \item{${J}$ is $\lambda^+$-complete;}
   \item{each $X\in {J}^+$ has a partition into $\lambda$ disjoint sets, each in ${J}^+$;}
   \item{there is a family $\mathcal{F}\subseteq {J}^+$ such that:
       \begin{enumerate}
       \item for each $X \in {J}^+$ there exists a $Y\in\mathcal{F}$ such that $Y\subseteq X$ ($\mathcal{F}$ is dense), and
       \item every descending $\omega$-sequence of sets from $\mathcal{F}$ has a nonempty intersection.
       \end{enumerate} }
   \end{enumerate}
 \end{quote}
 
If we make the convention that for a cardinal $\lambda$ the symbol $\lambda^+$ denotes the least \emph{infinite} cardinal greater than $\lambda$ (so that $2^+=\aleph_0$), then $I(2,\kappa)$ implies 
 that there is an ideal on $\kappa$ which satisfies Baumgartner's condition, whence ${\sf U}(\aleph_0,\kappa)$ holds. 

\begin{theorem}[Laver] \label{magidor}
  For infinite cardinals $\kappa$ and $\lambda$, $\emph{\textsf{I}}(\lambda,\kappa)$ implies ${\sf U}(\lambda,\kappa)$.
\end{theorem}

\begin{proof}
Let ${J}$ be an ideal witnessing $\textsf{I}(\lambda,\kappa)$. Define a strategy $F$ for White in ${\sf UG}(\lambda,\kappa)$ as follows:
As $\kappa$ is a member of ${J}^+$, choose a partition $\kappa = \bigcup_{\xi<\lambda}X_{\xi}$ such that each $X_{\xi}$ is a member of ${J}^+$. 
Define White's first move, $F(\emptyset)$, to be this partition. 

  When Black responds with $\xi_0$ and a partition $\langle X_{\xi_0,\xi} \colon \xi<\lambda \rangle$ of $X_{\xi_0}$, 
White 
chooses
  $X_{\xi_0,\xi_1}\in{J}^+$, which exists by the
  $\lambda^+$-completeness of ${J}$, and $F_0\in \mathcal{F}$
  such that $F_0 \subseteq X_{\xi_0,\xi_1}$. White then plays a partition
  $\langle X_{\xi_0,\xi_1,\xi}\colon \xi<\lambda \rangle$ of $X_{\xi_0,\xi_1}$
  so that $F_0\cap
  X_{\xi_0,\xi_1,\xi}\in{J}^+$ for all $\xi<\lambda$. 
  Next Black responds with $\xi_2$ and $\langle X_{\xi_0,\xi_1,\xi_2,\xi}
  \colon \xi<\lambda \rangle$, etc. We end up with a descending $\omega$-sequence
  \begin{equation*}
    X_{\xi_0}\supseteq\cdots\supseteq X_{\xi_0,\xi_1,\ldots,\xi_{n}}\supseteq\cdots
  \end{equation*}
with the properties that, for all $n$, 
  \begin{equation*}
    X_{\xi_0,\xi_1,\ldots,\xi_{2n+1}}\supseteq F_n \text{ and } F_n\cap
    X_{\xi_0,\xi_1,\ldots,\xi_{2n+1},\xi_{2n+2}} \in{J}^+.
  \end{equation*}

Since the sets $F_n$ are decreasing and in ${\mathcal F}$, we can find $\alpha\in\bigcap_{n<\omega} F_n$. Then $\alpha\in
  X_{\xi_0,\xi_1,\ldots,\xi_{2n+1}}$ for all $n$, and thus this is a
  winning play for White.
\end{proof}

For cardinals $\lambda\ge 2$ and $\kappa\ge 1$ we now have the following implications:
\[
 \textsf{I}(\lambda,\kappa) \stackrel{\mbox{\tiny Th \ref{magidor}}}{\Longrightarrow}  {\sf U}(\lambda,\kappa) \stackrel{\mbox{\tiny Th \ref{thm:ugtolaver}}}{\Longrightarrow} {\sf L}(\lambda,\kappa)\stackrel{\mbox{\tiny Th \ref{thm:lavertoGS}}}{\Longrightarrow}  {\sf D}(\lambda,\kappa)\stackrel{\mbox{\tiny Th \ref{BMGSconnection}}}{\Longrightarrow} {\sf B}(\lambda,\kappa)
\]
This sequence of implications permits us to establish via Theorem \ref{thm:getilambdakappa} the consistency, modulo the consistency of the existence of certain large cardinals, of Conjecture \ref{galvin2}, and thus of Conjecture \ref{galvin1}. The proof of Theorem \ref{thm:getilambdakappa} appears in large part in the proof of Theorem 4 of \cite{GJM}. 

Towards the proof of Theorem \ref{thm:getilambdakappa} recall: For a filter $\mathcal{F}$ on a set $S$, $\mathcal{F}^{+}$ denotes the set $\{Y\subseteq S: (\forall X\in\mathcal{F})(X\cap Y\neq \emptyset)\}$.
The filter $\mathcal{F}$ on a cardinal $\kappa>\aleph_0$ is said to be \emph{normal} if for each regressive function $f$ on a set $X \in \mathcal{F}^{+}$ there is a set $Y\subseteq X$ such that $Y$ is in $\mathcal{F}^{+}$ and $f$ is constant on $Y$.
If a cardinal number $\kappa$ is measurable, then it carries a normal $\kappa$-complete nonprincipal ultrafilter. 

\begin{theorem}[Laver] \label{thm:getilambdakappa}
Let $\lambda$ be an infinite cardinal, and let $\kappa$ be a measurable cardinal with $\kappa>\lambda$. L\'evy collapse $\kappa$ to size $\lambda^{++}$. In the resulting model there is an ideal on $\kappa=\lambda^{++}$ witnessing ${\sf I}(\lambda,\kappa)$.
\end{theorem}

\begin{proof}
  In $V$, let $\mathcal{U}\subseteq\mathcal{P}(\kappa)$ be a normal 
  ultrafilter on $\kappa$. In the generic extension, the witnessing ideal ${\mathcal I}$ 
  will be the dual of the filter that is generated by the ground model normal ultrafilter $\mathcal{U}$.

Let ${\mathbb P}=\Coll(\lambda^+,<\!\kappa)$ be the L\'evy collapse. Then ${\mathbb P}$ has the $\kappa$-chain condition and is $\lambda^+$-closed. Let $G$ be a ${\mathbb P}$-generic filter over $V$. In $V\lbrack G\rbrack$, let 
 $${\mathcal H}=\{X\subseteq\kappa\colon\exists U\in\mathcal{U}\,(U\subseteq X)\} $$
be the filter generated by ${\mathcal U}$.

Lemmas \ref{Lcoll1} and \ref{Lcoll2} give the details of Lemma 1 of \cite{GJM}.
  
 \begin{lemma}\label{Lcoll1}
${\mathcal H}$ is $\lambda^+$-complete.
  \end{lemma}

\begin{proof}
Working in $V$, let $\dot{\mathcal H}$ be a ${\mathbb P}$-name for the filter ${\mathcal H}$ of the extension, and let $(\tau_{\alpha}\colon\alpha<\lambda)$ be a ${\mathbb P}$-name for a sequence of elements of ${\mathcal H}$.

Consider any $q\in\mathbb{P}$. For each $\alpha<\lambda$, choose a
$q_{\alpha}\in\mathbb{P}$ and a $U_{\alpha}\in\mathcal{U}$ such that
for $\alpha<\beta<\lambda$ we have $q_{\beta}<q_{\alpha} < q$ and
$q_{\alpha}\forces U_{\alpha}\subseteq \tau_{\alpha}$. Since
$\mathbb{P}$ is $\lambda^+$-closed, fix a $p\in{\mathbb P}$ such
that for all $\alpha<\lambda$ we have $p<q_{\alpha}$. Also put $U =
\bigcap_{\alpha<\lambda}U_{\alpha}$. Since $\mathcal{U}$ is $\kappa$-complete and $\kappa>\lambda$,  $U$ is a member of $\mathcal{U}$. Then
  \[
  p \forces \forall
  \alpha<\lambda\,(U\subseteq\tau_{\alpha})
  \]
  and therefore
  \[
  p \forces \bigcap_{\alpha<\lambda}\tau_{\alpha} \in \dot{\mathcal H}.
  \]
Thus, for each $q\in{\mathbb P}$ there is a $p<q$ in ${\mathbb P}$ that forces that $\dot{\mathcal{H}}$ is $\lambda^+$-complete. 
\end{proof}

Call a set $X\subseteq \kappa$ an $\mathcal{H}$-\emph{positive} set if for each $U\in \mathcal{U}$ we have $U\cap X\neq \emptyset$. 

\begin{lemma}\label{Lcoll2}
${\mathcal H}$ is normal.
\end{lemma}

\begin{proof}
Suppose otherwise. Then choose an ${\mathcal H}$-positive set $X\subseteq \kappa$ 
and a regressive function $f:X\to\kappa$ that is not constant on any ${\mathcal H}$-positive set. Then for each $\alpha<\kappa$ there is a $U\in{\mathcal U}$ such that $U\cap f^{-1}(\{\alpha\})=\emptyset$.
 
Choose a condition $q\in G$ that, in $V$, forces the above of ${\mathbb P}$-names $\dot X$ and $\dot f$ for $X$ and $f$. Work in $V$.
For each $\alpha<\kappa$, choose a maximal antichain
$ A_{\alpha}=\{p^{\alpha}_{\gamma}\colon\gamma<\beta\} \subseteq{\mathbb P} $
below $q$ such that there is a sequence 
$ (U^{\alpha}_{\gamma}\colon\gamma<\beta) $
of elements of $\mathcal{U}$ with the property that, for each $\gamma<\beta$, we have
\[
   p^{\alpha}_{\gamma} \forces U^{\alpha}_{\gamma}\cap  \dot{f}^{-1}(\{\alpha\}) = \emptyset.
\]
Since $\mathcal{U}$ is $\kappa$-complete and ${\mathbb P}$ has the $\kappa$-chain condition, the set $U_{\alpha} = \bigcap_{\gamma<\beta}U^{\alpha}_{\gamma}$ is a member of $\mathcal{U}$, and each member $p$ of $A_{\alpha}$ satisfies
 $ p\forces U_{\alpha}\cap\dot{f}^{-1}(\{\alpha\}) = \emptyset $
and, therefore, the same is forced by $q$.  

We may assume that for each $\alpha<\beta<\kappa$, the
relation $U_{\beta}\subseteq U_{\alpha}$ holds. By normality of
$\mathcal{U}$, the set $U = \Delta_{\alpha<\kappa}U_{\alpha}$
is in $\mathcal{U}$, so 
$ q\forces \dot X\cap U\neq \emptyset. $ 

Consider $\eta\in U$. For all $\alpha<\eta$, we have
$\eta\in U_{\alpha}$ and so 
$ q\forces\dot{f}(\eta)\ne \alpha. $
But this contradicts that $f$ is regressive.
\end{proof}

Now Lemma 2 of \cite{GJM} proves
\begin{lemma}\label{densefamily}
The ground model filter $\mathcal{U}\subseteq \mathcal{I}^+$ is a $\sigma$-complete dense family in $\mathcal{I}^+$.
\end{lemma}

Since $\kappa=\lambda^{++}$, fix for each $\alpha<\kappa$ an injection $f_{\alpha}:\alpha\to\lambda^+$. Let $\text{Z}$ be an element of ${\mathcal I}^+$ and, for each $\eta<\lambda^+$ and each $\mu<\kappa$, define
\[
   S^{\mu}_{\eta}=\{\alpha\in \text{Z}\colon f_{\alpha}^{-1}(\eta)=\mu\}. 
\]
\begin{lemma}\label{sizes2} 
For some $\eta<\lambda^+$, we have
  $\vert\{\mu<\kappa\colon S^{\mu}_{\eta}\in {\mathcal I}^{+}\}\vert = \kappa$.
\end{lemma}

\begin{proof}
  For $\eta<\lambda^+$ put $T_{\eta} = \{\mu<\kappa\colon S^{\mu}_{\eta} \in {\mathcal I}^{+}\}$. If for each such $\eta$ we have $\vert T_{\eta}\vert<\kappa$, then
  $\vert\bigcup_{\eta<\lambda^{+}}T_{\eta}\vert<\kappa$. Consider
  $\beta\in\kappa\setminus\bigcup_{\eta<\lambda^{+}}T_{\eta}$. Since $\text{Z}$ is a member of ${\mathcal I}^{+}$, it has cardinality $\kappa$. We may assume that each element of $\text{Z}$ is larger than $\max\{\beta,\lambda^+\}$ .

Consider an $\alpha\in \text{Z}$. Now $f_{\alpha}(\beta)$ is defined and less than
$\lambda^+<\alpha$. Thus, $f:\text{Z}\to\kappa$ defined by $f(\alpha)=f_{\alpha}(\beta)$ is a regressive function. 
Then there is an $\eta<\lambda^+$ such that $f^{-1}(\{\eta\})$ is in ${\mathcal I}^{+}$. But then, since $\beta=f_{\alpha}^{-1}(\{\eta\})$ for any $\alpha\in f^{-1}(\{\eta\})$, we have that $S^{\beta}_{\eta}$ is ${\mathcal H}$-positive, and $\beta\in T_{\eta}$, a contradiction. 
\end{proof}

It follows that the ideal ${\mathcal I}$ dual to ${\mathcal H}$ witnesses ${\sf I}(\lambda,\kappa)$: $\lambda^+$ completeness holds by Lemma \ref{Lcoll1}. That each ${\mathcal I}$-positive set can be partitioned into $\lambda$ many positive sets holds by Lemma \ref{sizes2}.  
Finally, ${\mathcal U}$ plays the role of ${\mathcal F}$ in the definition of ${\sf I}(\lambda,\kappa)$. To see this, note first that ${\mathcal U}$  is closed under decreasing $\omega$-sequences, since ${\mathbb P}$ is $\lambda^+$-closed in $V$.
\end{proof}

Here is a scenario under which Conjecture \ref{galvin2} would be true:
\begin{lemma}\label{scenario1}
Assume that there is a proper class {\sf K} of cardinal numbers such that for each $\lambda\in {\sf K}$ there is a cardinal number $f(\lambda)$ such that ${\sf I}(\lambda,f(\lambda))$ holds. Then Conjecture \ref{galvin2} is true.
\end{lemma} 
\begin{proof}
If $\lambda^{\prime}\le \lambda$ and $\kappa$ are cardinal numbers, then $\textsf{I}(\lambda,\kappa)$ implies $\textsf{I}(\lambda^{\prime},\kappa)$
\end{proof}

\begin{corollary}\label{congalvin1} 
If there is a proper class ${\sf K}$ of cardinal numbers such that for each $\lambda\in{\sf K}$ there is a cardinal number $f(\lambda)$ such that ${\sf I}(\lambda,f(\lambda))$ holds, then Conjecture \ref{galvin1} is true. 
\end{corollary}
\begin{proof} By Theorem \ref{BMGSconnection} and Lemma \ref{scenario1}. 
\end{proof}

\begin{theorem}[Laver]\label{getclassilambdakappa}
If it is consistent that there is a proper class of measurable cardinals, then it is consistent that there is a proper class of regular cardinals $\lambda$ such that ${\sf I}(\lambda,\lambda^{++})$ holds.
\end{theorem}
\begin{proof}  Here is an outline of the argument given by Laver in a personal communication dated May 2, 1995: ``Start with a class of measurable cardinals, none of which is a limit of the others. Collapse them by upwards Easton forcing to be the successors of regular cardinals. Note that the next measurable cardinal is still measurable when it is time to collapse it." We leave the details of the construction to the reader.
\end{proof}

As a result we find
\begin{corollary}\label{congalvin2} If it is consistent that there is a proper class of measurable cardinals, then it is consistent that Conjecture \ref{galvin2} holds.
\end{corollary}

\begin{corollary}\label{congalvin1} If it is consistent that there is a proper class of measurable cardinals, then it is consistent that Conjecture \ref{galvin1} holds.
\end{corollary}

\section{Examples of multiboard games}

In this section we give in {\sf ZFC} a number of \emph{ad hoc} illustrations of the phenomenon that ONE has a winning strategy in a multiboard version of an undetermined Gale--Stewart game.

\begin{theorem} \label{granthamJul14} 
  Let $0< \kappa\le \aleph_0$ be a cardinal number.
  \begin{enumerate}
  \item[\emph{(1)}]
      There is a set $A_{\kappa}\subseteq {}^{\omega}2$ such that
      ${\sf GS}_2(A_{\kappa},n)$ is undetermined for $0<n<\kappa$, while ONE
      has a winning strategy in ${\sf GS}_2(A_{\kappa},\kappa)$.
  \item[\emph{(2)}]
      There is a set $A_{\kappa}\subseteq {}^{\omega}\omega$ such
      that ${\sf GS}_{\omega}(A_{\kappa},n)$ is undetermined for
      $0<n<\kappa$, while ONE has a winning strategy in
      ${\sf GS}_{\omega}(A_{\kappa},\omega)$ which ensures that TWO wins on
      fewer than $\kappa$ boards. (Hence ONE has a winning strategy in ${\sf GS}_\omega(A_\kappa,\kappa)$.)
  \end{enumerate}
\end{theorem}
\begin{proof}
  As the proofs of (1) and (2) are similar, we prove only $(2)$. \\
  \underline{Proof of (2):}\\
 For $x=(x_n\colon n<\omega)$ in $^{\omega}\omega$, define $B(x) := \{x_{2n}+x_{2n+1}\colon x_{2n}>0\}.$
 For each $\mathcal{B}\subseteq \mathcal{P}(\omega)$, define 
 $Y(\mathcal{B}) :=\{x\in\,^{\omega}\omega\colon B(x)\not\in \mathcal{B}\}$. 

The set $A_{\kappa}$ used in the proof of (2) is, for an appropriately chosen $\mathcal{B}$, of the form $Y(\mathcal{B})$. In preparation for the construction of $A_{\kappa}$ we now introduce the following concepts:

For two infinite subsets $A$ and $B$ of $\omega$ we say that $A$ \emph{is interlaced with } $B$
if for some $k\in\omega$ we have
$$a_0< b_0< a_1< b_1<\cdots< a_n< b_n<\cdots$$
where $(a_n:n<\omega)$ and $(b_n:n<\omega)$ are the increasing enumerations of $A\setminus k$ and $B\setminus k$.
We write $I(A,B)$ to denote that $A$ is interlaced with $B$. Observe that this is a symmetric binary relation on the power set of $\omega$. 
For $A\subseteq \omega$, define 
 \[
    I(A) := \{B\subseteq\omega\colon I(A,B)\}. 
 \] 
 Note that if $A$ is finite then $I(A)$ is empty. Also note that if $I(A)$ is nonempty, then it has at least $\aleph_0$ elements. For example, if $A$ is the set of even numbers then $I(A)$ is countable.  If
$A=\{a_n:n\in\omega\}$ where $a_0< a_1<\cdots< a_n<\cdots$, and if
$a_{n+1}-a_n>1$ for all but finitely many $n$ and $a_{n+1}-a_n>2$ for
infinitely many $n$, then $I(A)$ has $2^{\aleph_0}$ elements.

For an ordinal $\alpha\le \omega$, let $G^{\alpha}$ be the following game structure of length $\omega$ between players ONE and TWO. First ONE 
chooses an $f_0\in{}^{\alpha}\omega$, then TWO  
chooses an $f_1\in{}^{\alpha}\omega$, then ONE  
chooses an $f_2\in{}^{\alpha}\omega$, then TWO  
chooses an $f_3\in{}^{\alpha}\omega$, and so on.

For a play $F = \langle f_n\colon n<\omega\rangle \in \,{}^{\omega}({}^{\alpha}\omega)$ of $G^{\alpha}$, define for each $i\in \alpha$:
  \[
  B^F_i := B(\, ( f_0(i), f_1(i), f_2(i), \cdots)\,) \; = \{f_{2n}(i)+f_{2n+1}(i)\colon f_{2n}(i)>0\}.
  \]

\begin{lemma} \label{lemma:sigma} 
There is a strategy $\sigma$ for ONE  
in $G^{\omega}$ such that, if $F$ is a $\sigma$-play of $G^{\omega}$, then $I(B^F_i,B^F_j)$ holds for all $i,j\in \omega$ with
  $i\neq j$.
\end{lemma}
\begin{proof}
Let $(j_n:n<\omega)=(0,1,0,1,2,0,1,2,3,0,1,2,3,4,0,\dots)$. ONE's 
strategy is to choose $f_{2n}$ so that $f_{2n}(i)=0$ for $i\ne j_n$, while $f_{2n}(j_n)=f_{2n-2}(j_{n-1})+f_{2n-1}(j_{n-1})+1$ if $n> 0$, and $f_0(j_0)=1$.


\end{proof}

\begin{lemma} \label{lemma:ij} 
Suppose that $0<\kappa\le\omega$. If  $\mathcal{B}\subseteq\mathcal{P}(\omega)$ is such that for any sequence $(B_i:i\in\kappa)$ of elements of $\mathcal{B}$ there are $i<j<\kappa$ for which $I(B_i,B_j)$ fails, then there is a strategy $\sigma$ for ONE  
in ${\sf GS}_{\omega}(Y(\mathcal{B}),\kappa)$ which ensures that TWO  
wins on fewer than $\kappa$ boards. 
\end{lemma}
\begin{proof}
Let $\kappa$ and $\mathcal{B}$ be as in the hypotheses. Let $\sigma$ be the strategy from Lemma \ref{lemma:sigma} for ONE.  
We show that $\sigma$ has the claimed property. For let $F$ be a $\sigma$-play. Consider the sequence $(B^F_i:i<\omega)$. For any $i\neq j$ we have $I(B^F_i,B^F_j)$. Thus, as $\kappa\le\omega$, for any $J\subseteq\omega$ with $\vert J\vert = \kappa$, considering $(B^F_i:i\in J)$, there is a $B^F_i$ not in $\mathcal{B}$, meaning $x\in Y(\mathcal{B})$ where $B_i^F = B(x)$.  In particular, fewer than $\kappa$ of the sets $B^F_i$ are members of $\mathcal{B}$ and thus fewer than $\kappa$ of the corresponding sequences are not members of $Y(\mathcal{B})$. It follows that if ONE  
plays the game ${\sf GS}_{\omega}(Y(\mathcal{B}),\kappa)$ according to strategy $\sigma$, then TWO  
wins on fewer than $\kappa$ boards.
\end{proof}

It follows that for $\mathcal{B}$ as in Lemma \ref{lemma:ij}, TWO 
does not have a winning strategy in ${\sf GS}_{\omega}(Y(\mathcal{B}))$. 
Therefore the games ${\sf GS}_{\omega}(Y(\mathcal{B}),n)$ for $0<n<\kappa$, if they are not wins for ONE, are undetermined.

Towards the rest of the proof of Theorem \ref{granthamJul14} (2) we now define:
For $A,\, B\subseteq \omega$ we say that $D(A,B)$ holds if either of $A$ or $B$ is finite, or else there are for each $k\in\omega$ elements $a_1 < a_2 < a_3$ of $A\setminus k$ for which  there is no element $b$ of $B$ such that $a_1 < b < a_3$. 
We obviously have:
\begin{lemma} \label{lemma:dab}
If $D(A,B)$ holds, then $I(A)\cap I(B)
  = \emptyset$.\qed
\end{lemma}

\begin{lemma} \label{lemma:emb}  
Let $0<n<\omega$ as well as a strategy $\sigma$ for ONE in $G^n$ be given. For each $t\in\,^{\omega}2$ there is a $\sigma$-play $F_t$ such that if $s,t\in{}^{\omega}2$ are different, then for all $i, j\in n$, $D(B^{F_s}_i,B^{F_t}_j)$ holds. \qed
\end{lemma}   
The proof of Lemma \ref{lemma:emb} is left to the reader.

\begin{lemma} \label{lemma:mbf} 
Suppose
  $\mathcal{A}\subseteq{\mathcal{P}}(\omega)$,
  $\vert\mathcal{A}\vert<2^{\aleph_0}$.  If $0<n<\omega$ and $\sigma$ is
  a strategy for ONE
  in $G^n$, then there is a $\sigma$-play $F$ of
  $G^n$ such that $I(B^F_i)\cap\mathcal{A} =\emptyset$ for all $i\in
  n$.
\end{lemma}

\begin{proof}
  Lemmas \ref{lemma:dab} and \ref{lemma:emb}.
\end{proof}

\begin{lemma} \label{lemma:key} Suppose
  $\mathcal{A}\subseteq{\mathcal{P}}(\omega)$,
  $\vert\mathcal{A}\vert<2^{\aleph_0}$.  If $0<n<\omega$ and $\sigma$ is
  a strategy for ONE 
  in $G^n$, then there is a
  $\mathcal{B}\subseteq\mathcal{P}(\omega)$ such that
  $\vert\mathcal{B}\vert \le n$, $I(B)\cap\mathcal{A}=\emptyset$ for
  all $B\in\mathcal{B}$, and $\sigma$ is not a winning strategy for ONE 
  in ${\sf GS}_{\omega}(Y(\mathcal{B}),n)$.
\end{lemma}

\begin{proof}
  Choose $F$ as in Lemma \ref{lemma:mbf} and let $\mathcal{B} =
  \{B^F_i\colon i\in n\}$.
\end{proof}

Now we can complete the proof of item (2) of Theorem \ref{granthamJul14}: \\
Let $ \{(n_{\xi},\sigma_{\xi})\colon \xi<2^{\aleph_0}\}$ enumerate the set
 $$\{\langle n,\sigma\rangle\colon 0<n<\omega \mbox{ and $\sigma$ a strategy for ONE 
 in $G^n$}\}. $$ 
 Using Lemma \ref{lemma:key}, construct a sequence
 $(\mathcal{B}_{\xi}\colon \xi<2^{\aleph_0})$ so that:
 \begin{itemize}
 \item{$\mathcal{B}_{\xi}\subseteq\mathcal{P}(\omega)$;}
 \item{ $\vert\mathcal{B}_{\xi}\vert \le n_{\xi}$;}
 \item{ $I(B_1,B_2)$ does not hold if $B_1\in\mathcal{B}_{\xi_1}$, $B_2 \in\mathcal{B}_{\xi_2}$, and $\xi_1\neq \xi_2$, and}
\item{$\sigma_{\xi}$ is not a winning strategy for ONE
in ${\sf GS}_{\omega}(Y(\mathcal{B}_{\xi}),n_\xi)$.}
 \end{itemize}
  For $0<\kappa\le \omega$ let
 \[
   A_{\kappa}=Y\bigl(\bigcup\{\mathcal{B}_{\xi}\colon n_{\xi}<\kappa\}\bigr). 
\]
Clearly, for $0<n<\kappa$, ONE cannot have a winning strategy in ${\sf GS}_{\omega}(A_{\kappa},n)$, and the rest follows from Lemma \ref{lemma:ij}.
\end{proof}

Consider next the Morton Davis game ${\sf G}^*(Y)$, see Davis \cite{davis}: In inning $n <\omega$ ONE first chooses a finite sequence $o_n$ of zeroes and ones; TWO
responds with $t_n\in\{0,\,1\}$. ONE wins a play if
\[
o_1{}^\frown (t_1){}^\frown o_2{}^\frown (t_2){}^\frown\cdots{}^\frown
o_n{}^\frown (t_n){}^\frown\cdots \in Y;
\]
else, TWO wins.  Evidently this can be coded as a Gale-Stewart game
${\sf GS}_{\omega}(S)$ for some $S\subseteq {}^{\omega}\omega$.

\begin{theorem}[Davis]\label{mdavis} 
  Let $Y\subseteq{}^{\omega}\{0,1\}$ be given. Then TWO has a winning strategy in ${\sf G}^*(Y)$ if, and only if,  $Y$ is countable, and ONE has
  a winning strategy in ${\sf G}^*(Y)$ if, and only if, $Y$ contains a nonempty
  perfect subset.\qed
\end{theorem}

\begin{theorem} \label{1976} 
  There is a set $S\subseteq{}^{\omega}\omega$ such that for
  $1\le\kappa<2^{\aleph_0}$ the game ${\sf GS}_{\omega}(S,\kappa)$ is
  undetermined.
\end{theorem}

\begin{proof}
  If $2^{\aleph_0}=\aleph_1$, the result follows from Lemma
  \ref{grantham1}. Thus we may assume that
  $2^{\aleph_0}>\aleph_1$. Choose a set $Y\subseteq {}^{\omega}\{0,1\}$
  with $\aleph_0<\vert Y\vert<2^{\aleph_0}$. Since $Y$ is uncountable,
  Theorem \ref{mdavis} implies that TWO has no winning strategy in the
  game ${\sf G}^*(Y)$.

  We show that for $\kappa<2^{\aleph_0}$, ONE has no winning strategy in
the $\kappa$-board game  ${\sf G}^*(Y,\kappa)$.

  Fix $\kappa<2^{\aleph_0}$, and consider any strategy $\sigma$ for
  ONE. For $\alpha<\kappa$ and $t\in{}^{\omega}\{0,1\}$, consider
  the play of ${\sf G}^*(Y,\kappa)$ in which ONE follows $\sigma$ and, for all $n<\omega$,
in inning $n$, TWO plays $t(n)$ on each board. For each
  $\alpha$, let $f_{\alpha}(t)$ be the element on board $\alpha$
  resulting from ONE applying the strategy $\sigma$, and TWO following
  $t$. For fixed $\alpha$, the function
  \[
  f_{\alpha}:{}^{\omega}\{0,1\}\rightarrow{}^{\omega}\{0, 1\}
  \]
  is one-to-one. Thus for each $\alpha$, we have
  $\vert\{t\in{}^{\omega}\{0,1\}\colon f_{\alpha}(t)\in Y\}\vert\le\vert
  Y\vert$. But then we have
  $$ \vert\bigcup_{\alpha\in\kappa}\{t\in{}^{\omega}\{0,1\}\colon f_{\alpha}(t)\in
  Y\}\vert\le\kappa\cdot\vert Y\vert < 2^{\aleph_0}. $$

  Thus, if TWO follows $q\in{}^{\omega}\{0,1\}\setminus
  \bigl(\bigcup_{\alpha\in\kappa}\{t\in{}^{\omega}\{0,1\}\colon f_{\alpha}(t)\in
  Y\}\bigr)$, this defeats ONE's strategy $\sigma$.
\end{proof}

\begin{lemma} \label{grantham1}  There is a set $S\subseteq\ ^\omega\{0,1\}$ such that the game ${\sf GS}_2(S,\kappa)$ is undetermined for every nonzero cardinal $\kappa$
such that $2^\kappa\le2^{\aleph_0}$.
\end{lemma}

\begin{proof}
The construction is a modified version of the usual diagonalization
argument showing that there is some $S$ with ${\sf GS}_2(S)$ undetermined.

Let $\text{H}(S,\kappa)$ be the variant of ${\sf GS}_2(S,\kappa)$ in which
TWO is required to make the same move on all $\kappa$ boards. Thus, at move
$n$,
ONE chooses a function $f_n:\kappa\to\{0,1\}$, and then TWO chooses a
$\textit {constant}$ function $g_n:\kappa\to\{0,1\}$. It will suffice to
construct a set $S\subseteq\ ^\omega\{0,1\}$ so that TWO has no winning
strategy in ${\sf GS}_2(S)$ while, for each nonzero cardinal $\kappa$ with
$2^\kappa\le2^{\aleph_0}$, ONE has no winning strategy in
$\text{H}(S,\kappa)$.

Observe that, if  $2^\kappa\le2^{\aleph_0}$, then there are only
$2^{\aleph_0}$ strategies for ONE in $\text{H}(S,\kappa)$. Let
$\{\tau_\alpha:\alpha< 2^{\aleph_0}\}$ be the set of all strategies for TWO
in ${\sf GS}_2(S)$, and let
$\{(\kappa_\alpha,\sigma_\alpha):\alpha< 2^{\aleph_0}\}$ be the set of all
pairs $(\kappa,\sigma)$ where $\kappa$ is a nonzero cardinal with
$2^\kappa\le2^{\aleph_0}$ and $\sigma$ is a strategy for ONE in
$\text{H}(S,\kappa)$.

Let $\lambda=\text{cf }2^{\aleph_0}$. We are going to define, for each
$\alpha< 2^{\aleph_0}$, a point $s_\alpha\in\ \ ^\omega\{0,1\}$ and a set
$T_\alpha\subseteq\ ^\omega\{0,1\}$ with
$|T_\alpha|\le\kappa_\alpha< \lambda$. The point $s_\alpha$ is decreed to
be a winning play for ONE, and is chosen so as to defeat TWO's strategy
$\tau_\alpha$ in ${\sf GS}_2(S)$; the elements of $T_\alpha$ are decreed to
be winning for TWO, are are chosen so as to defeat ONE's strategy
$\sigma_\alpha$ in $\text{H}(S,\kappa_\alpha)$. Of course we require that
$s_\alpha\notin T_\beta$ for all $\alpha,\beta< 2^{\aleph_0}$.

Let $\alpha< 2^{\aleph_0}$ and suppose that $s_\beta$ and $T_\beta$ have
been defined for all $\beta <\alpha$.
Since $\alpha< 2^{\aleph_0}$ and $|T_\beta|< \lambda$ for all $\beta$, we
have $\lvert\bigcup_{\beta< \alpha}T_\beta\rvert<2^{\aleph_0}$. Inasmuch
as there are $2^{\aleph_0}$ different $\tau_\alpha$-plays of
${\sf GS}_2(S)$, we can choose a $\tau_\alpha$-play $s_\alpha$ which is not
in any $T_\beta$ with $\beta< \alpha$.

Next, choose a sequence $(\varepsilon_0,\varepsilon_1,\varepsilon_2,\dots)\in\ ^\omega\{0,1\}$ which differs from all the sequences $(s_\beta(1),s_\beta(3),s_\beta(5),\dots)$ for $\beta\le\alpha$. Let $(f_0,g_0,f_1,g_1,f_2,g_2,\dots)$ be the $\sigma_\alpha$-play of $\text{H}(S,\kappa_\alpha)$ in which TWO plays the constant function $g_n(\xi)=\varepsilon_n$ for every $n< \omega$. Let $T_\alpha$ consist of the sequences 
\[
   \begin{tabular}{lcl}
          $t_{\alpha,\xi}$ & = & $(f_0(\xi),g_0(\xi),f_1(\xi),g_1(\xi),f_2(\xi),g_2(\xi),\dots)$ \\
                                  &  = & $(f_0(\xi),\varepsilon_0,f_1(\xi),\varepsilon_1,f_2(\xi),\varepsilon_2,\dots)$
  \end{tabular}
\]
    for $\xi\in\kappa_\alpha$. Then $|T_\alpha|\le\kappa_\alpha$, and $s_\beta\notin T_\alpha$ for all $\beta\le\alpha$.

It is clear from the construction that the set $S=\{s_\alpha:\alpha< 2^{\aleph_0}\}$ has the desired properties.
\end{proof}

\section{Open problems}

It is consistent, relative to the consistency of the existence of a proper class of measurable cardinals, that Conjecture \ref{galvin2}, and thus Conjecture \ref{galvin1}, are true.

\begin{question} Is Conjecture \ref{galvin1} a theorem of {\sf ZFC}?
\end{question}

\begin{question} Is Conjecture \ref{galvin2} a theorem of {\sf ZFC}?
\end{question}

We saw also that if it is consistent that there is a measurable cardinal, then it is consistent that for each $S\subseteq\, ^{\omega}\omega$ the game ${\sf GS}_{\omega}(S, (2^{\aleph_0})^+)$ is determined. It is not clear if this result is optimal.

\begin{question} 
 Is there a set $S\subseteq {}^\omega\omega$ such that ${\sf GS}_{\omega}(S,2^{\aleph_0})$ is undetermined?
\end{question}

Recall the Morton Davis game ${\sf G}^*(S)$ introduced right before Theorem \ref{mdavis}. The $\kappa$-board version ${\sf G}^*(S,\kappa)$ was considered in the proof of Theorem \ref{1976}: 
\begin{question}
  Fix $\kappa$ with $1<\kappa\le 2^{\aleph_0}$. Characterize those
  subsets $S\subseteq {}^{\omega}\{0,1\}$ for which ONE has a winning strategy in ${\sf G}^*(S,\kappa)$.
\end{question}

We also do not know the answer to the following questions: 

\begin{question} 
 Is there a set $S\subseteq {}^\omega\omega$ such that ${\sf GS}_{\omega}(S,2^{\aleph_0})$ is determined but ${\sf GS}_{\omega}(S,\aleph_0)$ is undetermined?
\end{question}

\begin{question}
  If, for some cardinal $\lambda$, there is  a cardinal $\kappa$ such that ${\sf D}(\lambda,\kappa)$ holds, is then $\kappa\ge(2^\lambda)^+$?
\end{question}

\section*{Acknowledgements} As the reader will glean from the paper, we have benefited from the insights and communications of a large number of mathematicians in the development of the work reported in this paper. Some of them are no longer with us. Foremost we would like to acknowledge the fundamental contributions of Jim Baumgartner and Richard Laver. 
We would also like to acknowledge the contributions of M. Magidor, C. Gray, R.M. Solovay, J. Mycielski, A. Caicedo, R. Ketchersid, S. Hechler and R. McKenzie who have contributed to our understanding of the material presented here. We also thank the referee for a very careful reading of the paper and for remarks that improved the exposition in several respects.


\begin{thebibliography}{00}

\bibitem{davis}  M. Davis,  \emph{Infinite games of perfect information},  {\bf Advances in Game Theory, Ann. Math. Studies} {\bf 52} (1964), 85--101.

\bibitem{GJM}  F. Galvin,  T. Jech, and M. Magidor,  \emph{An ideal game},  {\bf The Journal of Symbolic Logic} { 43\,(2)} (1978), 284--292.

\bibitem{jech} 
T. Jech,  \emph{Set theory. The third millennium edition},  {\bf Springer}, Berlin, 2003.

\bibitem{JMP} T. Jech, M. Magidor, W. Mitchell and K. Prikry, \emph{Precipitous ideals}, {\bf The Journal of Symbolic Logic} 45:1 (1980), 1--8.

\bibitem{juhasz} I. Juh\'asz,  \emph{Cardinal functions in topology}, {\bf Mathematisch Centrum}, Amsterdam, 1971.       

\bibitem{scottishbook} 
R.D. Mauldin, (ed.).  \emph{The Scottish Book: Mathematics from the Scottish Caf\'e},  {\bf Birkh\"auser}, 1981.


\bibitem{Schreier} J. Schreier, \emph{Eine Eigenschaft abstrakter Mengen}, {\bf Studia Mathematicae} 7 (1938), 155--156.

\bibitem{Telgarsky} R. Telg\'arsky, \emph{Topological Games: On the 50$^{th}$ Anniversary of the Banach-Mazur Game}, {\bf Rocky Mountain Journal of Mathematics} 17 (1987), 227--276.

\bibitem{ulam} S. Ulam, \emph{Combinatorial analysis in infinite sets and some physical theories}, {\bf SIAM Review} 6 (1964), 343--355.


\bibitem{White} H.E. White, Jr., \emph{Topological spaces that are $\alpha$-favorable for a player with perfect information}, {\bf Proceedings of the American Mathematical Society} 50 (1975), 477--482.

{\flushleft
\begin{tabular}{ll}
\author{F. Galvin}                             & \hspace{0.75in}  \author{M. Scheepers}          \\
Department of Mathematics,            & \hspace{0.75in} Department of Mathematics,  \\ 
University of Kansas,                        & \hspace{0.75in}   Boise State University,         \\
Lawrence, KS 66045                        & \hspace{0.75in}   Boise, ID 83725                   \\ 
e-mail: bof@sunflower.com             & \hspace{0.75in}   e-mail:  mscheepe@boisestate.edu \\
\end{tabular}}


\end{thebibliography}
\end{document}